\documentclass[10pt]{amsart}
\usepackage{amssymb}
\usepackage{graphicx}
\usepackage[all,cmtip]{xy}
\renewcommand{\t}{\tau}

\newcommand{\D}{\mathbb{D}}

\renewcommand{\P}{\mathbb{P}}

\newcommand{\SA}{{\mathcal{A}}}

\newcommand{\SC}{{\mathcal{C}}}

\newcommand{\SJ}{{\mathcal{J}}}

\newcommand{\J}{\mathfrak{j}}

\newcommand{\Z}{\mathbb{Z}}
\newcommand{\C}{\mathbb{C}}

\renewcommand{\H}{\mathbb{H}}
\newcommand{\Q}{\mathbb{Q}}

\newcommand{\R}{\mathbb{R}}

\renewcommand{\S}{\mathbb{S}}
\renewcommand{\SC}{\mathcal{C}}
\renewcommand{\SA}{\mathcal{A}}

\newcommand{\coker}{\operatorname{coker}}

\newcommand{\im}{\operatorname{im}}

\newcommand{\rk}{\operatorname{rk}}

\newcommand{\End}{\operatorname{End}}

\newcommand{\Cont}{\operatorname{Cont}}
\newcommand{\Diff}{\operatorname{Diff}}

\newtheorem{proposition}{Proposition}
\newtheorem{theorem}[proposition]{Theorem}
\newtheorem{definition}[proposition]{Definition}
\newtheorem{lemma}[proposition]{Lemma}

\newtheorem{corollary}[proposition]{Corollary}
\newtheorem{remark}[proposition]{Remark}

\begin{document}

\title{A remark on the Reeb Flow for Spheres}



\author{Roger Casals}
\address{Instituto de Ciencias Matem\'aticas -- CSIC.
C. Nicol\'as Cabrera, 13--15, 28049, Madrid, Spain.}
\email{casals.roger@icmat.es}

\author{Francisco Presas}
\address{Instituto de Ciencias Matem\'aticas -- CSIC.
C. Nicol\'as Cabrera, 13--15, 28049, Madrid, Spain.}
\email{fpresas@icmat.es}



\begin{abstract}
We prove the non--triviality of the Reeb flow for the standard contact spheres $\S^{2n+1}$, $n\neq3$, inside the fundamental group of their contactomorphism group. The argument uses the existence of homotopically non--trivial $2$--spheres in the space of contact structures of a $3$--Sasakian manifold. 
\end{abstract}
\maketitle

\noindent Let $(M,\xi)$ be a closed contact manifold. Consider the space $\SC(M,\xi)$ of contact structures isotopic to $\xi$. This space has been studied in special cases. See \cite{El} for the $3$--sphere and \cite{Bo}, \cite{Ge} for torus bundles. In the present note we prove the non--triviality of its second homotopy group for $3$--Sasakian manifolds, see \cite{BG}. 
\begin{theorem}\label{thm:lin}
Let $(M,\xi)$ be a $3$--Sasakian manifold, then $\rk(\pi_2(\SC(M,\xi)))\geq1$.
\end{theorem}
\noindent Let $(\S^{4n+3},\xi_0=\ker\alpha_0)$ be the standard contact sphere with the standard contact form. The non--trivial spheres in $\SC(\S^{4n+3},\xi_0)$ allow us to answer a question posed in \cite{Gi}:\\

\noindent {\it \noindent Remarque 2.10: On peut se demander s'il n'y a pas, dans $\Cont(\S^{2n+1},\xi_0)$, un lacet positif contractile plus simple que dans $\P U(n,1)$ et par exemple si le lacet $\rho_t$, $t\in \S^1$, n'est pas contractile. C'est peu probable mais je n'en ai pas la preuve.}\\

\noindent The answer we provide is the following
\begin{corollary}\label{cor:Gi}
The class in $\pi_1(\Cont(\S^{2n+1},\xi_0))$ generated by the Reeb flow of $\alpha_0$ is a non--trivial element of infinite order for $n\neq3$.
\end{corollary}

\noindent In Section \ref{sec:pre} we introduce the objects of interest and necessary notation. The geometric construction underlying the results is explained in Section \ref{sec:lin}. It is a generalization to higher dimensions of ideas found in \cite{Ge}. Theorem \ref{thm:lin} is concluded. Section \ref{sec:gi} contains the argument deducing Corollary \ref{cor:Gi}. Section \ref{sec:gener} extends the results to higher homotopy groups.\\

\noindent{\bf Acknowledgements}: We are grateful to V. Ginzburg for useful discussions. The first author thanks O. Sp\'a$\check{\mbox{c}}$il for valuable remarks.

\section{Preliminaries}\label{sec:pre}
\subsection{Contact structures}
\begin{definition}
Let $M^{2n+1}$ be a smooth manifold. A codimension--$1$ regular distribution $\xi$ is a contact distribution if there exists a $1$--form $\alpha \in \Omega^1(M)$ such that $\ker \alpha = \xi$ and $\alpha \wedge d\alpha^{n}$ is a volume form.
\end{definition}
\noindent The structure described above is known as a cooriented contact structure. Since the non--coorientable case is not considered in this article, we refer to a cooriented contact structure simply as a contact structure. The smooth manifold $M$ will be assumed to be oriented. The contact structures to be considered will be positively cooriented, i.e. the induced orientation coincides with that prescribed on $M$.\\

\noindent The definition is independent of the choice of $1$--form $\alpha'= e^f \alpha$, for any $f\in C^{\infty}(M, \R)$. Let $\Cont(M,\xi)=\{s\in\Diff(M):ds_*\xi=\xi\}$ be the space of diffeomorphisms that preserve the contact structure. These diffeomorphisms are called contactomorphisms. The connected component of the identity of $\Cont(M,\xi)$ will be denoted by $\Cont_0(M,\xi)$. $\SC(M,\xi)$ will stand for the space of positive contact structures in $M$ isotopic to $\xi$. The unique vector field $R$ such that
$$i_R \alpha=1,\quad i_R d\alpha=0,$$
is called the Reeb vector field associated to $\alpha$.\\

\noindent A vector field $X\in\Gamma(TM)$ preserves the contact structure if it satisfies the following pair of equations
\begin{eqnarray*}
i_X \alpha & = & H, \\
i_X d\alpha &= & -dH + (i_R dH)\alpha,
\end{eqnarray*}
for a choice of $\alpha$ and a function $H\in C^{\infty}(M,\R)$. Such a function is called the Hamiltonian associated to the vector field. This correspondance defines a linear isomorphism between the space of vector fields $\Gamma_{\xi}(TM)$ preserving the contact structure $\xi$ and the vector space of smooth functions $C^{\infty}(M, \R)$. By definition, a contactomorphism $\phi \in \Cont_0(M, \xi)$ admits an expression as $\phi=\phi_1$ for a time dependent flow $\{ \phi_t \}_{t\in[0,1]}$ generated by a time dependent family $X_t\in \Gamma_{\xi}(TM)$. Therefore, its flow $\{{\phi}_t\}$ can be generated by a time dependent family of smooth functions $\{H_t \}$.\\

\subsection{Contact fibrations}
\noindent A smooth fibration $\pi: X \longrightarrow B$ is said to be contact for a codimension--$1$ distribution $\xi \subset TX$ if for any fiber $F_p= \pi^{-1}(p)\stackrel{e}{\hookrightarrow}X$, the restriction of the distribution $e^* \xi$ is a contact structure on the fibre. We assume that the distribution $\xi$ is cooriented. Any $\alpha\in\Omega^1(X)$ such that $\xi=\ker\alpha$ will be referred to as a fibration form.\\

\noindent Let $\pi:X \longrightarrow B$ be a smooth fibration. The vertical subbundle $V\subset TX$ is defined fiberwise by $V_x= \ker d\pi(x), \forall x\in X$. An {\em Ehresmann connection} is a smooth choice of a fiberwise complementary linear space $H_x$ for $V_x$ inside $T_xX$. Therefore, the map $d\pi_x: H_x \longrightarrow TB_{\pi(x)}$ is a linear isomorphism and there is a well-defined notion of parallel transport. \\

\noindent There is a canonical connection once a contact fibration $(\pi,\xi=\ker \alpha)$ is fixed. The connection $H$ is defined at a point $x\in X$ to be the annihilator of the vector subspace $V_x \cap \xi_x$ with respect to the quadratic form $(\xi, d\alpha)$. It is complementary to $V_x$ since $V_x \cap \xi_x$ is a symplectic space for the $2$--form $d\alpha$.  The connection is independent of the choice of fibration form $\alpha$. See \cite{Pr} for details on the following facts.
\begin{lemma}
The parallel transport of the canonical connection associated to a contact fibration is by contactomorphisms.
\end{lemma}
\begin{lemma} \label{lem:parallel}
Let $(F, \ker \alpha_0)$ be a closed contact manifold. Let $\pi: F\times\D^2 \longrightarrow \D^2$ be a contact fibration with fibration distribution defined by the kernel of $\alpha=\alpha_0 + Hd\theta$, for some function $H:F\times\D^2 \longrightarrow \R$ satisfying $|H|=O(r^2)$. Fix a loop $\gamma:\S^1 \longrightarrow\D^2$, defined as $\gamma(\theta)=\gamma(r_0, \theta)$ in polar coordinates. Then, the contactomorphism of the fiber $F \times (r_0, 0)$ defined by the parallel transport along $\gamma$ is generated by the family of Hamiltonian functions $\{G_{\theta}(p)=- H(p,r_0, \theta) \}_{\theta \in [0, 2\pi]}$.
\end{lemma}
\noindent Let us study general contact fibrations over a $2$--disk $\D^2$. Fix a contact fibration $\pi:X \longrightarrow\D^2$ with distribution $\xi=\ker\alpha$. Consider the radial vector field $Y= \partial_r$, defined on $\D^2 \setminus \{0\}$. It can be lifted to $X$ by using the canonical contact connection. This produces a vector field $\widetilde{Y}:X\setminus F_0\longrightarrow TX$. Once an angle $\theta_0$ is fixed it can be uniquely extended to $0\in\D^2$. In such a case, denote by $\phi_{r,\theta_0}: F_0 \longrightarrow F_{(r, \theta_0)}$ the associated flow at time $r$. It identifies via contactomorphisms the fibers over $0\in\D^2$ and over $(r,\theta_0)\in\D^2$. Define the diffeomorphism:
\begin{eqnarray*}
\Phi: F_0 \times D^2 & \longrightarrow & X \\
(p,r, \theta) & \longmapsto & \phi_{r, \theta} (p).
\end{eqnarray*}
Then the definition of the contact connection implies $\Phi^* \alpha = e^g(\alpha_0 +H d\theta)$, where $g:M \times D^2 \longrightarrow \R$ and $H: M \times D^2 \longrightarrow \R$ are arbitrary smooth functions. We can choose as fibration form $\alpha'= e^{-g} \alpha$ and trivialize the fibration using $\Phi$. Then we obtain the expression
\begin{equation}
\Phi^* \alpha' = (\alpha_0 +H d\theta). \label{eq:radial_triv}
\end{equation}
Given a contact fibration over the disk, the trivialization constructed above is called {\it radial}. It is convenient to observe that the radial trivialization construction can be made parametric for families of contact fibrations over the disk.

\subsection{Loops at infinity}
\noindent Fix a contact fibration $\pi:X \longrightarrow \S^2$ with distribution $\xi$, fibre $F$ and a point $N\in  \S^2$. This point will be referred to as North pole or infinity. Define the restriction fibration $\pi_N: X \setminus \pi^{-1}(N) \longrightarrow \S^2 \setminus N \simeq\D^2$. Trivialize the contact fibration $\pi_N$ {\it radially} from $S= \{ 0 \} \in\D^2$ to obtain a new contact fibration $\hat{\pi}: F \times\D^2\longrightarrow\D^2$ with fibration form $\alpha_0+ Hd\theta$. Denoting by $\Phi: F\times\D^2 \longrightarrow X\setminus \pi^{-1}(N)$ the trivialization map, we obtain $\Phi^* \xi = \ker\{\alpha_0+ Hd\theta\}$. Therefore, the map is connection--preserving. Consider the family of loops
\begin{eqnarray*}
\gamma_r: \S^1 & \longrightarrow & D^2 \\
\theta & \longmapsto & (r, \theta).
\end{eqnarray*}
Composing with the embedding $\D^2 \hookrightarrow \S^2$, for $r \longrightarrow1$, they are smaller and smaller loops around the North pole $N \in \S^2$. By Lemma \ref{lem:parallel}, the parallel transport associated to the loop $\gamma_r$ is generated by a family of Hamiltonians $\{ÊG_{\theta}^r \}_{\theta \in \S^1}$, defined by $G_{\theta}^r(p)= -H(p,r, \theta)$. The limit function
$$G_{\theta}= \lim_{r\longrightarrow 1} G_{\theta}^r$$
exists because the connection associated to $\xi$ is a smooth connection over $\S^2$. It is clear that $\{ G_{\theta} \}$ defines a loop in $\Cont(M, \xi_0 = \ker \alpha_0)$. This will be called the {\it loop at infinity} associated to $(\pi, \xi)$. Continuous families of contact fibrations with marked fibre produce continuous families of loops at infinity.

\begin{definition} A contact sphere is a smooth map $e: \S^2 \longrightarrow\SC(M, \xi)$.
\end{definition}
\noindent There is a canonical contact fibration over $\S^2$ associated to any contact sphere $e$. It is defined as
$$X= M\times \S^2\longrightarrow\S^2,$$
with the distribution at $(p,z)\in M \times \S^2$ being $\xi^e(p,z)= e(z)_p \oplus T_z \S^2 \subset T_pM \oplus T_z \S^2$. \\

\noindent Denote by $C^{\infty}(\S^2, \SC(M, \xi))$ the space of smooth maps from $\S^2$ to $\SC(M,\xi)$. The smooth loop space of $\Cont_0(M,\xi)$ is denoted as $\Omega(\Cont_0(M,\xi), id)$.
\begin{lemma} \label{lem:north}
The previous construction induces a continuous map
$$ C^{\infty}(\S^2, \SC(M, \xi)) \longrightarrow \Omega(\Cont_0(M,\xi), id). $$
Therefore, it provides a morphism
$$ \pi_2(\SC(M, \xi)) \longrightarrow \pi_1(\Cont_0(M,\xi)). $$
\end{lemma}

\subsection{Homotopy sequence}
The group $\Diff_0(M)$ acts transitively on $\SC(M,\xi)$ because of Gray's Stability Theorem. It is a Serre fibration with homotopy fibre  $\Cont(M,\xi)\cap\Diff_0(M)$. This homotopy fibre might be disconnected. Its identity component is denoted by $\Cont_0(M,\xi)$. Hence the fibration induces a long exact sequence
\begin{equation} \label{eq:seq}
\ldots\longrightarrow\pi_2(\Diff_0(M))\longrightarrow\pi_2(\SC(M,\xi))\stackrel{\partial_2}{\longrightarrow} \pi_1(\Cont_0(M,\xi))\longrightarrow\pi_1(\Diff_0(M))\longrightarrow\ldots. 
\end{equation}
The map $\partial_2$ is the one provided by Lemma \ref{lem:north}. The study of this sequence will provide Corollary \ref{cor:Gi}. \\

\noindent Note that a geometric lifting map
\begin{equation}
\pi_j(\SC(M,\xi))\stackrel{\partial_j}{\longrightarrow} \pi_{j-1}(\Cont_0(M,\xi))
\label{eq:lifting}
\end{equation}
can be analogously defined. It provides a geometric representative of the connecting morphism. This generalizes the previous constructions. It will be used in Section \ref{sec:gener}.
\section{Spheres in $\SC(M,\xi)$}\label{sec:lin}
\subsection{Almost contact structures}
\noindent Let $M$ be an oriented $(2n+1)$--dimensional manifold. Denote by $Dist(M)$ the space of smooth codimension--$1$ regular cooriented distributions on $M$. Concerning orientations, an almost complex structure on a cooriented distribution will be positive if the induced orientation coincides with the prescribed one. Define the space of almost contact structures as
$$\SA(M)=\{(\xi,\J):\xi\in Dist(M),\J\in\End(\xi),\J^2=-\mbox{id},\J\mbox{ positive}\}.$$
Given a contact structure $\xi=\ker \alpha$, an almost complex structure $\J\in\End(\xi)$ is said to be compatible with $\alpha$ if it is compatible with the symplectic form on the symplectic space $(\xi,d\alpha)$. The space $\SA(M)$ has a subset defined by
$$\SA\SC(M, \xi)=\{(\eta,\J):\eta \in \SC(M, \xi),\J\in\End(\eta),\J^2=-\mbox{id}, \J \mbox{ compatible with }\alpha\mbox{ such that }\eta=\ker\alpha\}.$$
The space of almost complex structures compatible with a fixed symplectic form is contractible. Thus, the forgetful map $\SA\SC(M, \xi)\longrightarrow\SC(M, \xi)$ has a contractible homotopy fibre. Hence there exists a homotopy inverse $\i:\SC(M, \xi)\longrightarrow\SA\SC(M, \xi)$ provided by the choice of a compatible almost complex structure on the contact distribution.\\

\noindent Fix a point $p\in M$ and an oriented framing $\t:T_pM\stackrel{\simeq}{\longrightarrow}\R^{2n+1}$. 
Define the evaluation map
$$e_{(p,\tau)}:\SA(M)\longrightarrow\SA(\R^{2n+1}),\quad e_{(p,\tau)}(\xi, \J)= (\tau_{*}\xi_p, \tau_* \J_p).$$
This is a continuous map and thus induces $\widetilde{e}_{(p,\tau)}:\pi_2(\SA(M))\longrightarrow\pi_2(\SA(\R^{2n+1}))$. Therefore, we obtain
$$\varepsilon_{(p,\tau)}=\widetilde{e}_{(p,\tau)}\circ\i_*:\pi_2(\SC(M,\xi))\longrightarrow\pi_2(\SA(\R^{2n+1}))$$
\begin{lemma}
$\pi_2(\SA(\R^{2n+1}))\cong\Z$.
\end{lemma}
\begin{proof}
The space $\SA(\R^{2n+1})$ is isomorphic to the homogeneous space $SO(2n+1)/U(n)$. The standard inclusion $SO(2n)\longrightarrow SO(2n+1)$ descends to a map
$$SO(2n)/U(n)\longrightarrow SO(2n+1)/U(n)$$
with homotopy fibre $\S^{2n}$. The long exact sequence for a homotopy fibration implies that
$$\pi_2(SO(2n)/U(n))\cong\pi_2(SO(2n+1)/U(n)),\quad n\geq2.$$ It is simple to show that $SO(2n+1)/U(n)$ is also isomorphic to $SO(2n+2)/U(n+1)$. Since $SO(4)/U(2)$ is a $2$--sphere, the statement follows.
\end{proof}
\noindent Thus the evaluation map can be seen as an integer--valued map for $\pi_2(\SC(M,\xi))$.
\begin{lemma}
The map $\varepsilon_{(p,\tau)}:\pi_2(\SC(M,\xi))\longrightarrow\pi_2(\SA(\R^{2n+1}))$ is independent of the choice of $p$ and $\tau$.
\end{lemma}
\begin{proof}
\noindent Let $p,q\in M$ and $\tau_p$,$\tau_q$ be oriented framings of $T_pM$, $T_qM$ respectively. Consider a continuous path of pairs $\{(p_t,\tau_t)\}$ connecting $(p,\tau_p)$ and $(q,\tau_q)$. The continuous family of maps
$$e_{(p_t,\tau_t)}:\SA(M)\longrightarrow\SA(\R^{2n+1}),\quad e_{(p_t,\tau_t)}(\xi, \J)= (\tau_{t_{*}}\xi_p, \tau_{t_*} \J_p)$$
provides a homotopy between $e_{(p,\tau_p)}$ and $e_{(q,\tau_q)}$.
\end{proof}
\subsection{Linear Contact Spheres}
\begin{definition}
A linear contact sphere is a contact sphere $\iota:\S^2\longrightarrow\SC(M,\xi)$ such that there exist three contact forms $(\alpha_0,\alpha_1,\alpha_2)$ satisfying
$$\iota(p)=\ker(e_0\alpha_0+e_1\alpha_1+e_2\alpha_2)$$
for the standard embedding $(e_0,e_1,e_2):\S^2\longrightarrow\R^3$.
\end{definition}
\begin{remark}
Such spheres can only exist in a $(4n+3)$--dimensional manifold. The fact that $\alpha$ and $-\alpha$ do not induce the same volume form in dimensions congruent to $1$ modulo $4$ yields an obstruction for their existence.
\end{remark}
\noindent Note that for a $3$--fold the triple $(\alpha_0,\alpha_1,\alpha_2)$ constitutes a framing of the cotangent bundle.
\begin{lemma}
Let $M$ be a $3$--fold and $S$ a linear contact sphere. The class $[S]\in\pi_2(\SC(M,\xi))$ is non--trivial and has infinite order.
\end{lemma}
\begin{proof}
Let $p\in M$ be a point and consider the framing $\tau=(\alpha_0,\alpha_1,\alpha_2)_p$. In the three--dimensional case $\SA(\R^3)$ is homotopic to a $2$--sphere. This homotopy can be realized by projection $\pi$ onto the space of cooriented $2$--plane distributions. The degree of the evaluation map is computed via
$$\S^2\stackrel{\varepsilon_{(p,\tau)}}{\longrightarrow}\SA(T_pM)\stackrel{\pi}{\longrightarrow}Dist(\R^3)\cong\S^2$$
$$z\longmapsto e_0(z)\alpha_0(p)+e_1(z)\alpha_1(p)+e_2(z)\alpha_2(p)\longmapsto (e_0(z),e_1(z),e_2(z)).$$
Being the identity, this map has degree $1$.
\end{proof}
\subsection{$3$--Sasakian manifolds}
Let us define a class of contact manifolds with natural linear contact spheres.

\begin{definition}
Let $(M^{4n+3},g)$ be a Riemannian manifold. It is said to be $3$--Sasakian if the holonomy group of the metric cone $(C(M),\bar{g})=(M\times\R^+, r^2g+dr\otimes dr)$ reduces to $Sp(n+1)$.
\end{definition}
\noindent This implies that $(C(M),\bar{g})$ is a hyperk\"ahler manifold $(C(M),\bar{g},I,J,K)$. The hyperk\"ahler structure induces a $2$--sphere of complex structures
$$\S^2(\bar{g})=\{e_0I+e_1J+e_2K: e_0^2+e_1^2+e_2^2=1\}.$$

\noindent Any such complex structure $\J\in\S^2(\bar{g})$ endows $(M\times\R^+,\bar{g})$ with a K\"ahler structure, providing $(M,g)$ with a Sasakian structure. The vertical vector field $\partial_r$ on $M\times\R^+$ is orthogonal to $M\times\{1\}$ and the form $\alpha$ defined by $\alpha_\J(v)=g(v,\J\partial_r)$ is a contact structure. Thus, a $3$--Sasakian structure provides a linear contact sphere $\{\alpha_\J\}_{\J\in\S^2(\bar{g})}$ generated by $\alpha_I,\alpha_J$ and $\alpha_K$.




\begin{theorem}\label{thm:linSas}
Let $M^{4n+3}$ be a $3$--Sasakian manifold. The class of the associated linear contact sphere is an element of infinite order in $\pi_2(\SC(M,\ker(\alpha_I)))$.
\end{theorem}
\begin{proof}
Let $p\in M$ and note that the $4n$--distribution $\eta=\ker\alpha_I\cap\ker\alpha_J\cap\alpha_K$ is $(I,J,K)$--invariant. Thus, it can be identified with the quaternionic vector space $\H^n$ by fixing a quaternionic framing $v=\{v_1,\ldots,v_n\}$. This induces a real framing $\tau=\{v,Iv,Jv,Kv\}$ for $\eta$, identifying it with $\R^{4n}$ endowed with the standard quaternionic structure.\\

\noindent Consider the Reeb vector fields $R_I,R_J,R_K$ associated to $\alpha_I,\alpha_J,$ and $\alpha_K$. Extend the framing $\tau$ to $\widetilde\tau=\{\tau,R_I,R_J,R_K\}$. Interpret the space $\SA(\R^{4n+3})$ as pairs of $(v,\J)$, where $v\in\S^{4n+2}\subset\R^{4n+3}$ is a unit vector and $\J$ an almost complex structure in its orthogonal space. Define
\begin{equation}\label{eq:h}
h: \SA(\R^{4n+3})\longrightarrow\mathcal{J}(\R^{4n+3}\oplus\R),\quad (v,\J)\longmapsto\{\widetilde{\J}:\langle v\rangle^{\perp}\oplus\langle v\rangle\oplus\langle\partial_t\rangle\longrightarrow \langle v\rangle^{\perp}\oplus\langle v\rangle\oplus\langle\partial_t\rangle\}
\end{equation}
where the almost complex structure is $\widetilde{\J}=\left(\begin{array}{ccc}
\J & 0 & 0\\
0 & 0 & -1\\
0 & 1 & 0
\end{array}\right)$. This induces a morphism of second homotopy groups. Through the above identification the linear contact sphere generated by $(\alpha_I,\alpha_J,\alpha_K)$ evaluates in a sphere $\langle(\xi_I,I),(\xi_J,J),(\xi_K,K)\rangle\in\SA(\R^{4n+3})$. This sphere maps via (\ref{eq:h}) to the sphere of complex structures generated by the triple $(I,J,K)$ in $\mathcal{J}(\R^{4n+4})$.\\

\noindent It is left to prove that the class of that sphere is an infinite order element of $\pi_2(SO(4n+4)/U(2n+2))$. Let us write $m=n+1$ to ease the notation. The homotopy fibration
$$U(2m)\longrightarrow SO(4m)\longrightarrow SO(4m)/U(2m)$$
induces an injection $\pi_2(SO(4m)/U(2m))\longrightarrow\pi_1(U(2m))\cong\Z$.\\

\noindent Let $(\theta,\phi)\in[0,2\pi]\times[0,\pi]$ be spherical angles. Define
$$J_\theta=\cos\theta J+\sin\theta K,\quad \widetilde{I}=\cos\phi I+\sin\phi J_\theta,\quad P_{\theta,\phi}=\cos(\phi/2)I+\sin(\phi/2)J_\theta.$$
The sphere is represented by $\widetilde{I}$, we shall compute its image under the boundary morphism. Note that $P_{\theta,\phi}\in SO(4m)$ and $\widetilde{I}=P_{\theta,\phi}^tIP_{\theta,\phi}$ . Further $P_{\theta,\pi}=J_\theta=(cos\theta\cdot id+\sin\theta I)J$, with $cos\theta\cdot id+\sin\theta I\in U(2m)$ and $J\in SO(4m)$. This decomposition provides a representative in $\pi_2(SO(4m)/U(2m))$. Thus the loop in $\pi_1(U(2m))$ is provided by $cos\theta\cdot id+\sin\theta I$ with $\theta\in[0,2\pi]$. Since the identification $\pi_1(U(2m))\cong\Z$ is given by the complex determinant, the degree of the sphere is $2m$.
\end{proof}
\noindent The argument above applies to a broader class of manifolds:
\begin{definition}
A contact manifold $(M,\xi_0)$ is said to possess an almost--quaternionic sphere if it admits a sphere $\S^2\stackrel{\xi}{\longrightarrow}\SC(M,\xi_0)$ such that:
\begin{itemize}
\item[1)] There exists a family ${\{\J_p\}}_{p\in\S^2}$ compatible with the contact distributions $\xi_p=\xi(p)$,
\item[2)] There exists a point $q\in M$ and a framing $\tau$ for $T_qM$ such that $e_{q,\tau}(\xi(\S^2))$ becomes the linear sphere associated to $\langle(\xi_I,I),(\xi_J,J),(\xi_K,K)\rangle\in\SA(\R^{4n+3})$.
\end{itemize}
\end{definition}
\begin{corollary}
An almost--quaternionic sphere inside a contact manifold $(M,\xi)$ generates a class of infinite order in $\pi_2(\SC(M,\xi))$.
\end{corollary}
\section{Reeb Flow for Spheres}\label{sec:gi}

\noindent Let us prove Corollary \ref{cor:Gi}. The standard contact sphere will be denoted $(\S^{2n+1},\xi)$. The relevant case is that of the spheres $\S^{2k+1}$ with $k$ odd. Indeed, for the spheres $\S^{2k+1}$ with $k=2n$ the Reeb flow is non--trivial in $\pi_1(SO(4n+2))\hookrightarrow\pi_1(\Diff_0(\S^{4n+1}))$. Thus it cannot be contractible in $\Cont_0(M,\xi)\subset\Diff_0(\S^{4n+1})$. In order to conclude the case $\S^{4n+3}$ we detail the construction in Sections \ref{sec:pre} and \ref{sec:lin}.\\

\noindent Consider the endomorphisms $I,J,K$ of $\R^{4(n+1)}$ obtained by direct sum of the corresponding endomorphisms $i,j,k$ of $\R^4$, satisfying the quaternionic relations
$$i^2=j^2=k^2=ijk=-1.$$
The endomorphisms $I,J,K$ anti--commute and hence any of their linear combinations is a complex structure. Let $e=(e_0,e_1,e_2):\S^2\longrightarrow\R^3$ be the standard embedding of the $2$--sphere in Euclidean $3$--space with azimuthal angle $\theta$ and polar angle $\phi$:
$$e_0=\cos\theta\sin\phi,\quad e_1=\sin\theta\sin\phi,\quad e_2=\cos\phi,\quad (\theta,\phi)\in[0,2\pi]\times[0,\pi].$$

\noindent A complex structure $\J\in End(\R^{4n+4})$ induces the real $(4n+2)$--distribution
$$\xi_\J=T\S^{4n+3}\cap\J T\S^{4n+3}$$ of $\J$--complex tangencies on the sphere $\S^{4n+3}$. There exists a unique, up to scaling, $U(\J,n)$--invariant $1$--form $\alpha_\J$ such that $\ker\alpha_\J=\xi_\J$. It is given by $\alpha(z)=z^t\J dz$. We use the following three $1$--forms
$$\alpha_0=\alpha_I,\quad\alpha_1=\alpha_J,\quad \alpha_2=\alpha_K.$$
Their respective Reeb vector fields $R_0$, $R_1$ and $R_2$ are linearly independent and their flows are given by the family of rotations generated by $I$, $J$ and $K$. Consider the $1$--form $\alpha=e_0\alpha_0+e_1\alpha_1+e_2\alpha_2$. The form $\alpha$ is a contact form on $\S^{4n+3}$ for each value of $e$. Although not used in the rest of the article, it is simple to prove the following
\begin{lemma}
$(\S^2\times\S^{4n+3},\ker\alpha)$ is a contact manifold.
\end{lemma}

\noindent Let us compute the loop at infinity for the trivial contact fibration
$$\S^2\times\S^{4n+3}\longrightarrow\S^{2}, (e,p)\longmapsto e.$$
In the spherical coordinates above, we will obtain the loop at infinity corresponding to $\phi=\pi$. The contact connection allows us to lift a vector field $X$ in the base $\S^2$. The lift $\widetilde{X}$ is the unique vector field on $\S^2\times\S^{4n+3}$ conforming the two conditions
$$\alpha(\widetilde{X})=0,\qquad d\alpha(\widetilde{X},V)=0,\quad\mbox{with }V\mbox{ an arbitrary vertical vector field}.$$
Since uniqueness is provided once a solution is found, the following assertion can be readily verified
\begin{lemma}
The lift of the polar vector field $\partial_\phi$ to the contact connection given by $\alpha$ is
$$\widetilde{X}_\phi=\partial_\phi+\frac{1}{2}\left(-\sin\theta R_0+\cos\theta R_1\right).$$
\end{lemma}
\noindent The Hamiltonian will appear once we pull--back the contact form $\alpha$ with the $\pi$--time flow of the lift $\widetilde{X}_\phi$. Consider the linear endomorphism $F_\theta=\frac{1}{2}\left(-\sin\theta I+\cos\theta J\right)$. The flow associated to $\widetilde{X}_\phi$ induces a diffeomorphism between the central fibre $\{\phi=0\}$ and the fibre at an arbitrary $\phi$. This diffeomorphism can be expressed as
$$\varphi_\phi:\S^{4n+3}\longrightarrow\S^{4n+3},\quad \varphi(p)=e^{F_\theta\phi}p.$$
This is understood as a map in complex space $\C^{2n+2}$ restricted to the sphere. The theory explained in Section \ref{sec:pre}, in particular formula (\ref{eq:radial_triv}), implies that the pull--back will be of the form $\alpha_2+H(p,\phi)d\theta$. A computation yields
\begin{lemma}
$\varphi_\phi^*(\alpha)=\alpha_2+\sin^2(\phi/2)d\theta$
\end{lemma}
\noindent The loops correspond to the flow of the vector field associated to $G=-\sin^2(\phi/2)$. The loop at infinity has Hamiltonian $G|_{\phi=\pi}\equiv-1$. Thus it is the Reeb flow.\\

\noindent We have geometrically realized the boundary map of the long exact homotopy sequence (\ref{eq:seq}). The non--contractibility of the Reeb flow will follow from an understanding of the contact sphere above and the group $\pi_2(\Diff_0(\S^{4n+3}))$. Regarding the former we have the following

\begin{lemma}\label{lem:acs} Let $S$ be the sphere of complex structures
$$S=\{e_0I+e_1J+e_2K:e\in\S^2\}\subset SO(4n+4)/U(2n+2).$$
\begin{itemize}
\item[1)] $S$ represents a non--trivial element of $\pi_2(SO(4n+4)/U(2n+2))\cong\Z$.
\item[2)] The image of $S$ in $\SC(\S^{4n+3},\xi)$ generates an infinite cyclic subgroup in $\pi_2(\SC(\S^{4n+3},\xi))$.
\end{itemize}
\end{lemma}
\begin{proof}
Both statements follow from the argument provided in the proof of Theorem \ref{thm:linSas}.
\end{proof}


\noindent Concerning the group $\Diff_0(\S^{4n+3})$, the following lemma will suffice.
\begin{lemma}\label{lem:tor}
$\pi_2(\Diff(\S^{4n+3}))\otimes\Q=0$ for $n\geq 2$.
\end{lemma}
\begin{proof}
This is a result in algebraic topology. Let $\Diff_0(\D^l,\partial)$ be the group of diffeomorphisms of the $l$--disk restricting to the identity at the boundary. Note the homotopy equivalence
$$\Diff_0(\S^l)\simeq SO(l+1)\times\Diff_0(\D^l,\partial)$$
and that $\pi_2(SO(l+1))=0$ since $SO(l+1)$ is a Lie group. Let $\phi(l)=\min\{(l-4)/3,(l-7)/2\}$. In the stable concordance range $0\leq j<\phi(l)$ we have
\begin{equation}\label{fh}
\pi_j(\Diff_0(\D^l,\partial))\otimes\Q=0,\quad\mbox{if }l\mbox{ even or }4\not|j+1.
\end{equation}
See \cite{We} for details. In particular $\pi_2(\Diff_0(\D^l,\partial))\otimes\Q=0$ for $l>11$. We are thus able to conclude
$$\pi_2(\Diff_0(\S^{4n+3}))\otimes\Q\cong\pi_2(\Diff_0(\D^{4n+3},\partial))\otimes\Q=0,\quad n>2.$$
For the case $n=2$ we provide a more {\it ad hoc} argument. Let $C(\D^{11})$ be the space of pseudo--isotopies for the disk $\D^{11}$. There exists a homotopy fibration
$$\Diff_0(\D^{12},\partial)\longrightarrow C(\D^{11})\longrightarrow\Diff_0(\D^{11},\partial)$$
Algebraic $K$--theory implies $\pi_1C(\D^{11})\otimes\Q=\pi_2C(\D^{11})\otimes\Q=0$. Observe that (\ref{fh}) implies that $\pi_1(\Diff(\D^{12},\partial))$ is a torsion group. The long exact homotopy sequence of the above fibration gives
$$\ldots\longrightarrow\pi_2(C(\D^{11}))\stackrel{\rho_2}{\longrightarrow}\pi_2(\Diff_0(\D^{11},\partial))\stackrel{\partial}{\longrightarrow}\pi_1(\Diff_0(\D^{12},\partial))\stackrel{i_1}{\longrightarrow}\pi_1(C(\D^{11}))\longrightarrow\ldots$$
This implies the short exact sequence of Abelian groups
$$0\longrightarrow A\longrightarrow\pi_2(\Diff_0(\D^{11},\partial))\longrightarrow B\longrightarrow0,$$
where $A=\ker\partial=\im\rho_2$ and $B=\im i_1=\coker\rho_2$. Thus $\pi_2(\Diff_0(\D^{11},\partial))$ is a torsion group.
\end{proof}
\begin{remark}
The Smale conjeture $\Diff_0(\S^3)\simeq SO(4)$ holds for $\S^3$, see \cite{Ha}.
\end{remark}


\noindent In order to conclude Corollary \ref{cor:Gi} for $\S^{4n+3}$ consider the class of the Reeb loop in $\pi_1(\Cont_0(\S^{4n+3},\xi))$. The construction explained above shows that it lies in the image of the boundary morphism
$$\partial_2:\pi_2(\SC(\S^{4n+3},\xi))\longrightarrow\pi_1(\Cont_0(\S^{4n+3},\xi)).$$
If the Reeb class were to be zero the sphere $S$ would lie in the image of $\pi_2(\Diff_0(\S^{4n+3}))$ in (\ref{eq:seq}). Lemma \ref{lem:tor} implies that such a sphere needs to be a torsion class if $n\geq 2$. Lemma \ref{lem:acs} contradicts this statement. Thus proving Corollary \ref{cor:Gi}.\\

\section{Higher homotopy groups} \label{sec:gener}
\noindent The previous arguments can be modified for $n$--dimensional homotopy spheres. This allows us to conclude properties of the higher homotopy type of the contactomorphism group. Consider the evalution map
$$
e_{p,\tau}: \SA(M) \longrightarrow \SA(\R^{2n+1}).
$$
Composition with the homotopy inverse $\i: \SC(M) \to \SA\SC(M)$
defines higher homotopy maps
$$
\pi_k(e_{p,\tau} \circ \i): \pi_{k}(\SC(M)) \longrightarrow \pi_k(\SA(\R^{2n+1})),\quad k \geq 1.
$$
Let us provide a simple application. Define the natural inclusion
$$
i_{\SJ}: \SJ(\R^{2n+2}) \longrightarrow \SC(\S^{2n+1},\xi),\quad\mbox i_{\SJ}(\J)= T\S^{2n+1} \cap \J T\S^{2n+1}.$$
\begin{lemma} \label{lem:inclu}
The map $i_{\SJ}$ is a homotopy inclusion.
\end{lemma}
\begin{proof}
Consider the following chain of maps
$$
c: \SJ(\R^{2n+2}) \stackrel{i_{\SJ}}{\longrightarrow} \SC(\S^{2n+1},\xi) \stackrel{e_{p,\tau} \circ \i}{\longrightarrow} \SA(\R^{2n+1}) \stackrel{h}{\longrightarrow} \SJ(\R^{2n+2}).
$$
The definition of each map implies $c= id$. Therefore, it induces the identity in homotopy:
$$
\pi_k(c)=id: \pi_k(\SJ(\R^{2n+2})) \stackrel{\pi_k(i_{\SJ})}{\longrightarrow} \pi_k(\SC(\S^{2n+1}),\xi) \stackrel{\pi_k(e_{p,\tau}\circ \i)}{\longrightarrow} \pi_k(\SA(\R^{2n+1})) \stackrel{\pi_k(h)}{\longrightarrow} \pi_k(\SJ(\R^{2n+2})).
$$
Thus the map $i_{\SJ}$ induces an injection $\pi_k(i_{\SJ})$, $\forall k\geq0$.
\end{proof}
\noindent This lemma can be combined with results on the homotopy type of the group $\Diff(\S^{2n+1})$. We can then conclude the existence of infinite order elements in certain homotopy groups of $\Cont(\S^{2n+1},\xi)$. Among many others, a simple instance is the following
\begin{lemma}
The group $\pi_5(\Cont(\S^{2n-1},\xi))$ has an element of infinite order, for $n \geq 12$.
\end{lemma}
\begin{proof}
Using the connecting map $\partial_6$, as described in equation (\ref{eq:lifting}), the statement is reduced to the following two assertions:
\begin{itemize}
\item[-] $\pi_6(\SJ(\R^{2n}))= \pi_6(SO(2n)/U(n))=\Z$ and therefore, by Lemma \ref{lem:inclu}, $\rk(\pi_6(\SC(\S^{2n-1},\xi)))\geq 1$.
\item[-] $\pi_6(\Diff(\S^{2n-1})) \otimes \Q=0$, for $n\geq 12$. This is again a consequence of the results in \cite{We}.
\end{itemize}
\end{proof}
\noindent Bott Periodicity Theorem allows us to apply the same argument to infinitely many other homotopy groups of $\Cont(\S^{2n-1},\xi)$. These techniques can be adapted for general contact manifolds as long as there is a partial understanding of the homotopy type of their group of diffeomorphisms.

\end{document}